\documentclass[12pt]{amsart}
\usepackage[english]{babel}
\usepackage{latexsym}
\usepackage{amsbsy}
\usepackage{amssymb}
\usepackage{amsfonts, amsmath}
\usepackage{mathtools}
\usepackage{csquotes}
\usepackage{pst-tree}
\usepackage{multirow}
\usepackage{blkarray}
\usepackage{tabularx}
\usepackage{arydshln,leftidx,mathtools}
\usepackage{ifthen}
\usepackage{tikz}
\usepackage{tikz-cd}
\usepackage{float}
\usepackage{picins}
\usepackage{amsbsy}
\usepackage{enumitem}
\usetikzlibrary{matrix}
\usetikzlibrary{calc}
\usetikzlibrary{fit}
\usepackage{amsmath,amsfonts,amssymb}
\usepackage{mathrsfs,eurosym}
\usepackage{pst-plot,pst-eucl}
\usepackage{amsfonts}
\usepackage{amssymb}
\usepackage{amsmath}

\usepackage{latexsym}
\usepackage{amsbsy}
\usepackage{alltt}
\usepackage{fontenc}
\usepackage{newlfont}
\usepackage{xlop}
\usepackage{graphicx}
\usepackage{yhmath}
\usepackage{mathdots}
\usepackage{MnSymbol}
\usepackage{amsthm}
\usepackage{url}

\newtheorem{thrm}{Theorem}[section]
\newtheorem{lem}[thrm]{Lemma}
\newtheorem{prop}[thrm]{Proposition}
\newtheorem{cor}[thrm]{Corollary}
\theoremstyle{definition}
\newtheorem{definition}[thrm]{Definition}
\newtheorem{remark}[thrm]{Remark}

\numberwithin{equation}{section}

\usepackage{cite}
\newcommand{\End}{\text\rm{End}}
\newcommand{\Inj}{\text\rm{Inj}}
\DeclareMathOperator{\Span}{span}
\numberwithin{equation}{section}

\title[Linear recurrence sequences]{Some algebraic results concerning linear recurrence sequences}
\author{\bf M. ~MOU\c{C}OUF}
\date{}
\begin{document}
\maketitle
\begin{center}
{\footnotesize Department of Mathematics, Faculty of Science,\\Chouaib Doukkali University, Morocco\\
Email: moucouf@hotmail.com}
\end{center}
\begin{abstract}
 We study the set $\mathcal{L}_{F}$ of all $F$-vector spaces $L(P)$ where $P$ is monic and splits over $F$ and $L(Q)$ denotes the set of linear recurrence sequences over $F$ with characteristic polynomial $Q$. We show that $\mathcal{L}_{F}$ can be endowed with two structures of graded commutative semiring. This study allows us to obtain, in compact forms, the polynomial $P,Q\in F[X]$ such that $L(P)=\prod_{i=1}^{m}L(P_{i})$ and $L(Q)=L(P_{1})\ast\cdots\ast L(P_{m})$, where $P_{1}, \ldots, P_{m}$
 are any monic polynomials over $F$.
\end{abstract}
%
%
\section{Introduction}
%
Let $F$ be a field and let $\mathcal{C}_{F}$ be the set of all sequences $\pmb{s}=(s_{n})_{n\geqslant 0}$ over $F$. Let us consider the following operations:
\begin{itemize}
\item componentwise sum $+$, defined by
$$\pmb{c}+\pmb{d}=\pmb{a},\quad a_{n}=c_{n}+d_{n},\quad \forall n\in \mathbb{N};$$
\item Hadamard product $.$, defined by
$$\pmb{c}\pmb{d}=\pmb{c}.\pmb{d}=\pmb{a},\quad a_{n}=c_{n}d_{n},\quad \forall n\in \mathbb{N};$$
\item Hurwitz product $\ast$, defined by
$$\pmb{c}\ast\pmb{d}=\pmb{a},\quad a_{n}=\sum_{i=0}^{n}\binom{n}{i}c_{i}d_{n-i},\quad \forall n\in \mathbb{N}.$$
\end{itemize}
It is well known that $(\mathcal{C}_{F},+,.)$ and $(\mathcal{C}_{F},+,\ast)$ are commutative rings with the same additive identity $\pmb{0}=(0,0,0\ldots)$
and with multiplicative identities $\pmb{1}=(1,1,1,\ldots)$ and
$\pmb{1^{\ast}}=(1,0,0,\ldots)$, respectively.
 ~\\ \indent
Let $\widetilde{\mathcal{C}_{F}}$ be the set of all $F$-subspaces of $\mathcal{C}_{F}$.It is easily seen that
\begin{enumerate}[label=(\alph*)]
 \item $(\widetilde{\mathcal{C}_{F}},+,.)$ is a commutative semiring, in which $\Span\{\pmb{0}\}$ and $\Span\{\pmb{1}\}$ are the additive identity and
 the multiplicative identity respectively, where $H+H'$ is the sum of the subspaces $H$ and $H'$, and $HH'$ is the subspace of $\mathcal{C}_{F}$ spanned
 by all Hadamard products $\pmb{c}\pmb{d}$ with $\pmb{c}\in H$ and $\pmb{d}\in H'$.
 \item $(\widetilde{\mathcal{C}_{F}},+,\ast)$ is a commutative semiring, in which $\Span\{\pmb{0}\}$ and $\Span\{\pmb{1^{\ast}}\}$ are the additive identity
 and the multiplicative identity respectively, where $H\ast H'$ is the subspace of $\mathcal{C}_{F}$ spanned
 by all Hurwitz products $\pmb{c}\ast\pmb{d}$ with $\pmb{c}\in H$ and $\pmb{d}\in H'$.
\end{enumerate}
~\indent It is well known that the set of all linear recurrence sequences can be endowed with several interesting algebraic
structures~\cite{Alecci, Cerrut, Chin, Kura1, Kura2, Lars, Peters}, and that the set $L(P)$ of all linear recurrence sequences over $F$ having $P$ as a characteristic polynomial is a $F$-vector space of dimension $\deg(P)$.
\\~\indent In this paper we discuss some algebraic properties of the subset $\mathcal{L}_{F}$ of $\widetilde{\mathcal{C}_{F}}$ consisting of all $L(P)$ where $P$ is monic and splits over $F$. it is well-known that For all $P,Q\in F[C]$, $L(P)+L(Q)=L(H)$ where $H$ is the least common multiple of $P$ and $Q$. From this it follows clearly that $(\mathcal{L}_{F},+)$ is a commutative monoid. In fact, we have more than this. we will show that $(\mathcal{L}_{F},+,.)$
and $(\mathcal{L}_{F},+,\ast)$ are, respectively, subsemirings of $(\widetilde{\mathcal{C}_{F}},+,.)$ and $(\widetilde{\mathcal{C}_{F}},+,\ast)$ graded by
the multiplicative monoid $F$. In order to prove these we define a new commutative semiring structure on the set of nonnegative integers $\mathbb{N}$.
The addition in $\mathbb{N}$ is defined to be the the maximum $\vee$ of two integers and the multiplication, $\wedge$, is the disjunction operation introduced
in~\cite{Zier} for positive integers and extended to $\mathbb{N}$ by setting $0\wedge n=n\wedge 0=0$ for all $n\in \mathbb{N}$.
~\\ \indent
In addition, we show, with the aid of this result and others presented in this paper, that if \begin{equation*}
P_{1}=X^{s_{1}}Q_{1},\ldots,P_{m}=X^{s_{m}}Q_{m}
\end{equation*}
are monic polynomials over $F$, where $Q_{i}(0)\neq 0,\, 1\leq i\leq m$, then
\begin{enumerate}[label=(\arabic*)]
\item \begin{equation*}
\prod_{i=1}^{m}L(P_{i})=L(X^{\rho}\Upsilon(Q_{1},\ldots,Q_{m})),
\end{equation*}
where
\begin{eqnarray*}
\rho=\begin{cases}
\min\{s_{i}/i\in \Theta\}&\quad\text{if}\quad \Theta\neq\emptyset\\
\max\{s_{i}/1\leq i\leq m\}& \quad\text{otherwise}
\end{cases}
\end{eqnarray*}
and
\begin{equation*}
\Theta=\{i/Q_{i}=1\}.
\end{equation*}
\item \begin{equation*} L(P_{1})\ast\cdots\ast L(P_{m})=L(\Psi(P_{1},\ldots,P_{m})).\end{equation*}
\end{enumerate}
$\Upsilon(Q_{1},\ldots,Q_{m})$ and $\Psi(P_{1},\ldots,P_{m})$ are polynomials over $F$ that we determine in terms of the roots of the polynomials
$P_{1},\ldots,P_{m}$ in an algebraic closure of $F$.
\\Results $(1)$ and $(2)$ above however are not new and, in many works, They are obtained under identical hypothesis or a bit more restrictive one using
either direct methods, or a Hopf Algebra approach.
\cite{Buck, Cak, Cerrut, Chin, Nied, Gott, Kura1, Kura2, Selm, Zier, Zya}. Our approach, however, gives more detailed information and allows obtaining
results $(1)$ and $(2)$ in a compact form to facilitate their application in future works. furthermore, we end the paper by discussing results obtained by Chin
and Goldman~\cite{Chin}. The authors of the aforementioned paper have derived several important results. However, in Remark~\ref{remer} we will see that
the main result in~\cite{Chin} is not correct, and also how this can be corrected.
\\~\indent It is known that there are various definitions of a semiring
in the literature. In this paper, We use that taken from the well-known book by J. Golan~\cite{Golan}.
A nonempty set $S$ together with two associative binary operations $+$ and $.$ will be called a commutative semiring provided:
\begin{enumerate}[label=(\roman*)]
\item $(S, +)$ is a commutative monoid with neutral element $0$ such that $0.a=0$ for all $a\in S$,
\item $(S, .)$ is a commutative monoid with neutral elemnt $1\,(\neq0)$,
\item $.$ is distributive over $+$.
\end{enumerate}
\subsection{Notations}\hfill\\
Throughout the paper, we use the following notations:
\begin{itemize}
\item $F$ is a field and $\mathcal{C}_{F}$ is the set of all sequences $\pmb{s}=(s_{n})_{n\geqslant 0}$ over $F$. It is well known that $\mathcal{C}_{F}$ is
an $F$-algebra under componentwise addition, multiplication and scalar multiplication;
\item $\widetilde{\mathcal{C}_{F}}$ be the set of all subspaces of $\mathcal{C}_{F}$;
\item Sequences in this paper are written in bold symbol;
\item $\binom{n}{i}$ denotes the binomial coefficient considered as an element of $F$;
\item $\pmb{\Lambda}_{i}$, $i\geqslant 0$, is the element $(\binom{n}{i})_{n\geq 0}$ of $\mathcal{C}_{F}$;
\item $\pmb{0}_{i}$ is the sequence of $\mathcal{C}_{F}$ defined by $\pmb{0}_{i}(n)=\delta_{in}$;
\item Let $\lambda \in F^{\ast}=F\backslash\{0\}$, then $\pmb{\lambda}$ is the geometric sequence $(\lambda^{n})_{n\geq 0}$;
\item Let $\lambda \in F^{\ast}=F\backslash\{0\}$ and $s\in \mathbb{N}$, then\\
    $\begin{array}{llll}
    \langle\lambda\rangle_{s}&=&\Span\{\pmb{\lambda}\Lambda_{0},\ldots,\pmb{\lambda}\Lambda_{s-1}\}&\quad\text{if}\quad s\neq0\\
    \langle0\rangle_{s}&=&\Span\{\pmb{0}_{0},\ldots,\pmb{0}_{s-1}\}&\quad\text{if}\quad s\neq0\\
    \langle\lambda\rangle_{0}&=&\langle0\rangle_{0}=\Span\{\pmb{0}\};&
    \end{array}$
\item $\pmb{0}=(0,0,0,\ldots)$;
\item $\pmb{1}=(1,1,1,\ldots)$;
\item $\pmb{1^{\ast}}=(1,0,0,\ldots)$;
\item Let $\lambda \in F^{\ast}$, then $I_{\lambda}=\{\langle\lambda\rangle_{s}/s\in \mathbb{N}\}$;
\item $I_{0}=\{\langle0\rangle_{s}/s\in \mathbb{N}\}$.
\end{itemize}
\section{The commutative semiring $(\mathbb{N},\vee, \wedge)$}
%
In this section we endow the set $\mathbb{N}$ of nonnegative integers with a binary operation $\wedge$ and we prove that $(\mathbb{N},\vee,\wedge)$ is a
commutative semiring, where $\vee$ is the the maximum of two integers.
\\We begin by recalling following definition which is from \cite{Gott}.
\begin{definition}\label{defini 1}
For two positive integers $i$ and $j$, let $i\wedge j$ be the maximum value of $r+s+1$ such that $\binom{r+s}{r}\neq0$ where
$r$ and $s$ are nonnegative integers with $0\leq r\leq i-1$ and $0\leq s\leq j-1$.
\end{definition}
It is straightforward that
\begin{equation*}
i\vee j\leq i\wedge j\leq i+j-1,
\end{equation*}
and that if $\text\rm{char}(F)=0$, then $i\wedge j=i+j-1$. More generally, we have
$$i\wedge j=i+j-1\quad\text{if and only if}\quad \binom{i+j-2}{i-1}\neq0.$$
Recall that in~\cite{Zier}, Zierler and Mills have defined a binary operation on the set $\mathbb{N^{\star}}$ of positive integers as follows:
\\Let $i$ and $j$ be any positive integers.
\\If $F$ is a field with characteristic $\text\rm{char}(F)=0$, then
 $$i\bigvee j=i+j-1 \,\,\text{for all}\,\, i,j\in \mathbb{N^{\star}}.$$
\\If $F$ is a field with characteristic $\text\rm{char}(F)=p>0$, we consider the $p$-adic digit expansions of $i-1$ and $j-1$
\begin{equation*}
i-1=\sum_{m\geq 0}i_{m}p^{m},
 \end{equation*}
 \begin{equation*}
j-1=\sum_{m\geq 0}j_{m}p^{m}.
\end{equation*}
Let $q(i,j)$ be the smallest nonnegative integer such that $i_{m}+j_{m}<p$ for all $m\geq q(i,j)$, then
\begin{align}\label{for 21}
i\bigvee j=p^{q(i,j)}+\sum_{m\geq q(i,j)}(i_{m}+j_{m})p^{m}.
\end{align}
Let $l\leq (i\bigvee j)-1$ be a nonnegative integer. It can be easily shown using Lucas' theorem that
there exist nonnegative integers $r\leq i-1$ and $s\leq j-1$ such that $l=r+s$ and $\binom{l}{r}\neq0$. Applying this result to
$l=(i\bigvee j)-1$ yields $(i\bigvee j)\leq i\wedge j$. Conversely, let $r\leq i-1$ and $s\leq j-1$ be two nonnegative integers such that $i\wedge j=r+s+1$ and
$\binom{r+s}{r}\neq0$. Let $r=\sum_{m\geq 0}r_{m}p^{m}$ and $s=\sum_{m\geq 0}s_{m}p^{m}$ be the $p$-adic digit expansions of $r$ and $s$, respectively.
Since $\binom{r+s}{r}\neq0$, it follows from Corollary~$4.72$ of~\cite{Ber} that $r_{m}+s_{m}<p$ for all $m\in \mathbb{N}$. Then
\begin{eqnarray*}
1+r+s&=&\sum_{m\geq q(i,j)}r_{m}p^{m}+\sum_{m\geq q(i,j)}s_{m}p^{m}+1+\sum_{m<q(i,j)}(r_{m}+s_{m})p^{m}\\
&\leq& \sum_{m\geq q(i,j)}i_{m}p^{m}+\sum_{m\geq q(i,j)}j_{m}p^{m}+p^{q(i,j)}\\
&=&i\bigvee j.
\end{eqnarray*}
This means that $i\wedge j\leq i\bigvee j$. Therefore, $i\wedge j= i\bigvee j$. This proved that the above definition given by Zierler and Mills
coincides with the definition~\ref{defini 1} recalled above and given in~\cite{Gott}, as mentioned by the authors of this last paper.
\\Later in their paper~\cite{Chin}, the authors adopted Zierler-Mills's defintion but in the following slightly different form
\begin{eqnarray*}
t(i,j)=\begin{cases}
i+j-1&\quad\text{if}\quad p=0\\
i+j-1&\quad\text{if}\quad p\neq0\,\,\text{and}\,\,q(i+1,j+1)=0\\
(i+1)\bigvee (j+1)-1& \quad\text{otherwise}
\end{cases}
\end{eqnarray*}
The following theorem is one of the keys of this paper.
\begin{thrm}\label{thm 221} Let $i,j\in \mathbb{N}$ and $\lambda, \mu\in F^{\ast}$. Put $0\wedge s=s\wedge 0=0$ for all $s\in \mathbb{N}$. Then
\begin{enumerate}[label=(\roman*)]
\item $i\wedge j$ is the unique integer such that $\langle1\rangle_{i}\langle1\rangle_{j}=\langle1\rangle_{i\wedge j}$.
\item $\langle\lambda\rangle_{i}\langle\mu\rangle_{j}=\langle\lambda\mu\rangle_{i\wedge j}$.
\end{enumerate}
\end{thrm}
\begin{proof}~
\begin{enumerate}[]
\item It is clair that if $i=0$ or $j=0$, then
\begin{equation*}
\langle1\rangle_{i}\langle1\rangle_{j}=\langle1\rangle_{0}=\langle1\rangle_{i\wedge j}.
\end{equation*}
Assume that $1\leq i\leq j$. Since
\begin{equation*}
\langle1\rangle_{i}=\Span\{\Lambda_{0},\ldots,\Lambda_{i-1}\}
\end{equation*}
and
\begin{equation*}
\langle1\rangle_{j}=\Span\{\Lambda_{0},\ldots,\Lambda_{j-1}\},
\end{equation*}
we have
\begin{equation*}
\langle1\rangle_{i}\langle1\rangle_{j}=\Span\{\Lambda_{e}\Lambda_{t}/0\leq e\leq i-1; 0\leq t\leq j-1\}.
\end{equation*}
Obviously $\Lambda_{0}\Lambda_{t}=\Lambda_{t}$ for all $0\leq e\leq i-1$ and all $0\leq t\leq j-1$, so
\begin{equation*}
\Lambda_{0},\ldots,\Lambda_{j-1}\in \langle1\rangle_{i}\langle1\rangle_{j},
 \end{equation*}
i.e.,
\begin{equation*}
\langle1\rangle_{j}\subseteq \langle1\rangle_{i}\langle1\rangle_{j}.
\end{equation*}
On the other hand, let $m$ be an integer such that $1\leq m\leq i\wedge j-1$. Then there exist $e_{m}\leq i-1$ and $t_{m}\leq j-1$ such that
$e_{m}+t_{m}=m$ and $\binom{m}{e_{m}}\neq 0$.
\\Now consider the following identities
$$\Lambda_{e_{m}}\Lambda_{t_{m}}=\sum_{q=e_{m}}^{m}\binom{q}{e_{m}}\binom{e_{m}}{q-t_{m}}\Lambda_{q},\,\, j\leq m\leq i\wedge j-1.$$
From the first identity and the fact that $\Lambda_{0}\ldots \Lambda_{j-1}\in \langle1\rangle_{i}\langle1\rangle_{j}$, we get that
$\Lambda_{j}\in \langle1\rangle_{i}\langle1\rangle_{j}$. Using this result and that
$\Lambda_{0}\ldots \Lambda_{j}\in \langle1\rangle_{i}\langle1\rangle_{j}$, from the second identity it follows that
$\Lambda_{j+1}\in \langle1\rangle_{i}\langle1\rangle_{j}$. We proceed in this way, successively, until we get that
$\Lambda_{i\wedge j-1}\in \langle1\rangle_{i}\langle1\rangle_{j}$. Consequently
$\langle1\rangle_{i\wedge j}\subseteq \langle1\rangle_{i}\langle1\rangle_{j}$.
\\ For the reverse inclusion, let $r\leq i-1$ and $s\leq j-1$ be two nonnegative integers. Then
$$\Lambda_{r}\Lambda_{s}=\sum_{q=r}^{r+s}\binom{q}{r}\binom{r}{q-s}\Lambda_{q}.$$
\\Consequently, in view of Definition~\ref{defini 1}, $\Lambda_{e}\Lambda_{t}\in \langle1\rangle_{i\wedge j}$ and hence
    $\langle1\rangle_{i}\langle1\rangle_{j}\subseteq \langle1\rangle_{i\wedge j}$. The uniqueness of $i\wedge j$
    follows immediately from the fact that $\dim(\langle1\rangle_{m})=m$ for all $m\in \mathbb{N}$.
\item Since $\langle\lambda\rangle_{i}=\langle\lambda\rangle_{1}\langle1\rangle_{i}$, it follows that
\begin{eqnarray*}
\langle\lambda\rangle_{i}\langle\mu\rangle_{j}&=&\langle\lambda\rangle_{1}\langle\mu\rangle_{1}\langle1\rangle_{i}\langle1\rangle_{j}\\
&=&\langle\lambda\mu\rangle_{1}\langle1\rangle_{i\wedge j}.
\end{eqnarray*} That is,
\begin{equation*}
\langle\lambda\rangle_{i}\langle\mu\rangle_{j}=\langle\lambda\mu\rangle_{i\wedge j}.
\end{equation*}
\end{enumerate}
\end{proof}
Now we can prove the following theorem.
\begin{thrm}\label{thm 1074}
$(\mathbb{N},\vee,\wedge)$ is a commutative semiring, where $\vee$ is the the maximum of two integers.
\end{thrm}
\begin{proof} Obviously $(\mathbb{N},\vee)$ is a commutative  monoid with neutral element $0$.
Let $i,j$ be any nonnegative integers. We have
\begin{equation*}
\langle1\rangle_{i}\langle1\rangle_{j}=\langle1\rangle_{j}\langle1\rangle_{i}
\end{equation*}
 and
 \begin{equation*}
 (\langle1\rangle_{i}\langle1\rangle_{j})\langle1\rangle_{s}=\langle1\rangle_{i}(\langle1\rangle_{j}\langle1\rangle_{s}),
 \end{equation*}
  then $i\wedge j=j\wedge i$ and $(i\wedge j)\wedge s=i\wedge(j\wedge s)$, that is, the operation $\wedge$ is commutative and associative. We also have
  \begin{equation*}
  \langle1\rangle_{i}\langle1\rangle_{1}=\langle1\rangle_{i}
  \end{equation*}
   then $i\wedge 1=i$ and $(\mathbb{N},\wedge)$ is a monoid with identity element $1$. It is obvious that $0$ is an absorbing element of $\wedge$. Finally,
   we have
   \begin{equation*}
   (\langle1\rangle_{i}+\langle1\rangle_{j})\langle1\rangle_{m}=\langle1\rangle_{i}\langle1\rangle_{m}+\langle1\rangle_{j}\langle1\rangle_{m},
    \end{equation*}
    which gives
    \begin{equation*}
    (i\vee j)\wedge m=(i\wedge m)\vee(j\wedge m);
    \end{equation*}
     thus $\wedge$ is distributive with respect to $\vee$. This completes the proof.
\end{proof}
\section{Some properties of the operation $\wedge$}
The operation $\wedge$ depends only on the characteristic of $F$ and has the following properties.
\begin{lem}\label{lem 197} Let $i,j$ be positive integers and $s,s',t,t'$ nonnegative integers. Let $p=\text\rm{char}(F)$. Then
\begin{enumerate}[label=(\arabic*)]
\item If $s'\leq s$ and $t'\leq t$, then $s'\wedge t'\leq s\wedge t$.
\item $i\wedge j$ is the unique integer such that $i\wedge j\leq i+j-1$, $\binom{(i\wedge j)-1}{i-1}\neq 0$ and $\binom{i\wedge j}{e}=0$ for all
integer $e$ (if it exists) such that
\begin{equation*}
(i\wedge j)-j+1\leq e\leq i-1
\end{equation*}
\end{enumerate}
\end{lem}
\begin{proof}~
\begin{enumerate}[label=(\arabic*)]
\item Follows immediately from the fact that $\langle1\rangle_{s'}\langle1\rangle_{t'}$ is a subspace of $\langle1\rangle_{s}\langle1\rangle_{t}$.
\item Since $i\wedge j=(i\wedge j)-e+e$ with $e\leq i-1$ and $(i\wedge j)-e\leq j-1$, it follows that $\binom{i\wedge j}{e}=0$ for all
integer $e$ such that
\begin{equation*}
(i\wedge j)-j+1\leq e\leq i-1.
\end{equation*}
To see that $\binom{(i\wedge j)-1}{i-1}\neq 0$ let $s$ and $t$ be two nonnegative integers such that $s\leq i-1$, $t\leq j-1$, $(i\wedge j)-1=s+t$ and
$\binom{s+t}{s}\neq0$.
\\If $s=i-1$, then there is nothing to prove. If $s< i-1$, then we have $$0=\binom{i\wedge j}{s+1}=\binom{(i\wedge j)-1}{s+1}+\binom{(i\wedge j)-1}{s},$$
and hence $\binom{(i\wedge j)-1}{s+1}\neq0$. Repeating this until we get $\binom{(i\wedge j)-1}{i-1}\neq0$.
\end{enumerate}
\end{proof}
\begin{remark}
For another proof of Part $(1)$, see Lemma~$2.1$ of~\cite{Chin}.
\end{remark}
We find the following result due to G\"{o}ttfert and Niederreiter~\cite{Nied}.
\begin{cor}\label{corrr}
Let $i$ and $j$ be any positive integers. Then
$i\wedge j=i+m-1$, where $m$ is the largest integer with $1\leq m\leq j$ such that $\binom{i+m-2}{i-1}\neq 0$.
\end{cor}
\begin{proof}
Let $m=(i\wedge j)-i+1$, then $i+m-2=(i\wedge j)-1$. Hence from Part $(2)$ of Lemma~\ref{lem 197} we have $\binom{i+m-2}{i-1}\neq 0$. Suppose that there
exists a positive integer $n$ such that $1\leq m+n\leq j$. Then, $(i\wedge j)-j+1\leq i-n\leq i-1$. Hence, in view of Part $(2)$ of Lemma~\ref{lem 197},
$$\binom{i\wedge j}{i-n}=\binom{i\wedge j}{i-n+1}=\cdots=\binom{i\wedge j}{i-1}=0.$$
Now using the Pascal formula
$$\binom{e}{t}=\binom{e-1}{t}+\binom{e-1}{t-1},$$
we find easily that
$$\binom{i\wedge j}{i-1}=\binom{(i\wedge j)+1}{i-1}=\cdots=\binom{(i\wedge j)+n-1}{i-1}=0.$$ The proof concludes by observing that
$i+(m+n)-2=(i\wedge j)+n-1$.
\end{proof}
The following corollary is a direct consequence of Lemma~\ref{lem 197} and Corollary~\ref{corrr}.
 \begin{cor}\label{cor1r}
Suppose $d=(i+1)\wedge(j+1)-1< n\leq i+j$, then
\begin{enumerate}[label=(\roman*)]
\item  $\binom{d}{i}\neq0$,
\item $\binom{n}{i}=0$.
\end{enumerate}
\end{cor}
\begin{proof} In view of Part $(2)$ of Lemma~\ref{lem 197}, $\binom{d}{i}=\binom{(i+1)\wedge(j+1)-1}{(i+1)-1}\neq0$ and the statement $(i)$ holds. Consider
now $n$ such that $$(i+1)\wedge(j+1)-1< n\leq i+j$$
then
$$(i+1)\wedge(j+1)-(i+1)+1< m=n+1-i\leq j+1,$$
which implies by Corollary~\ref{corrr}, that
$$\binom{(i+1)+m-2}{(i+1)-1}=0,$$
i.e.,
$$\binom{n}{i}=0.$$ Which completes the proof of Part $(ii)$.
\end{proof}
\section{The graded commutative semirings $(\mathcal{L}_{F},+,.)$ and $(\mathcal{L}_{F},+,\ast)$}
It is clair that $\langle\lambda\rangle_{s}+\langle\lambda\rangle_{t}=\langle\lambda\rangle_{s\vee t}$ for all $\lambda\in F$, $\vee$ is the the
maximum of two integers. Hence $I_{\lambda}=\{\langle\lambda\rangle_{s}/s\in \mathbb{N}\}$ is a submonoid of $(\widetilde{\mathcal{C}_{F}},+)$ and so is $\sum_{\lambda \in F^{\ast}}I_{\lambda}$. In view
of Theorem~2.1 of~\cite{Mouc},
$$\sum_{\lambda \in F^{\ast}}I_{\lambda}$$ is the internal direct sum of the monoids $I_{\lambda}, \lambda\in F^{\ast}$, and
$$I_{0}+\sum_{\lambda \in F^{\ast}}I_{\lambda}$$ is the internal direct sum of the monoids $I_{\lambda}, \lambda\in F$.
\\We shall now try to prove that $\mathcal{L}_{F}=I_{0}+\sum_{\lambda \in F^{\ast}}I_{\lambda}$ and that $(\mathcal{L}_{F},+,.)$ and $(\mathcal{L}_{F},+,\ast)$ are subsemirings of
$(\widetilde{\mathcal{C}_{F}},+,.)$ which are graded by the multiplicative monoid $F$.
\\
In the following lemma we introduce a binary operation on $\mathbb{N}$ and establish some of its properties that are needed in this section and the next.
\begin{lem}\label{lemzm} Let $\lambda\in F$ and let $\wedge_{\lambda}$ be the noncommutative binary operation on $\mathbb{N}$ defined by
\begin{equation*}
t\wedge_{\lambda} s=\min(t, s)\delta_{0,\lambda}+t\delta_{0,s}^{c}\delta_{0,\lambda}^{c}
\end{equation*}
where $\delta_{e,f}^{c}=1-\delta_{e,f}$ and $\delta_{e,f}$ is the Kronecker symbol. Then for all $s,t,e,f\in \mathbb{N}$ and all $\lambda\in F$ we have
\begin{enumerate}[label=(\roman*),resume]
\item If $e\leq f$ and $s\leq t$, then $e\wedge_{\lambda}s\leq f\wedge_{\lambda} t$.
\item $\langle0\rangle_{t}\langle\lambda\rangle_{s}=\langle0\rangle_{t\wedge_{\lambda} s}.$
\end{enumerate}
\end{lem}
\begin{proof} Part $(i)$ is obvious. Next we shall show Part $(ii)$. Since $\pmb{0}_{i}\Lambda_{j}=\binom{i}{j}\pmb{0}_{i}$ We have
$\langle0\rangle_{t}\langle1\rangle_{s}=\langle0\rangle_{t}$ for all $t\in \mathbb{N}$ and $s\in \mathbb{N}^{\ast}$. Then, for all $\lambda \in F^{\ast}$
and for all $(t,s)\in \mathbb{N}\times \mathbb{N}^{\ast}$,
\begin{eqnarray*}
\langle0\rangle_{t}\langle\lambda\rangle_{s}&=&\langle0\rangle_{t}\langle\lambda\rangle_{1}\langle1\rangle_{s}\\
&=&\langle\lambda\rangle_{1}\langle0\rangle_{t}\\
&=&\Span\{\pmb{0}_{0},\lambda\pmb{0}_{1},\ldots, \lambda^{t}\pmb{0}_{t}\}\\
&=&\langle0\rangle_{t}\quad (\text{since}\, \lambda\, \text{is invertible}).
\end{eqnarray*}
Hence, for all $\lambda \in F^{\ast}$ and all $t,s\in \mathbb{N}$, we have
\begin{equation*}
\langle0\rangle_{t}\langle\lambda\rangle_{s}=\langle0\rangle_{t\delta_{0,s}^{c}}
\end{equation*}
But since
\begin{equation*}
\langle0\rangle_{t}\langle0\rangle_{s}=\langle0\rangle_{\min(t, s)}
\end{equation*}
it follows that
\begin{equation*}
\langle0\rangle_{t}\langle\lambda\rangle_{s}=\langle0\rangle_{t\wedge_{\lambda} s}
\end{equation*}
for all $\lambda \in F$ and all $t,s\in \mathbb{N}$.
\end{proof}
\begin{prop} Let $F$ be a field. Then $\mathcal{L}_{F}^{\ast}=\sum_{\lambda \in F^{\ast}}I_{\lambda}$ the set of all $L(P)$ where $P$ is a monic polynomial
with all roots in $F^{\ast}$.
\end{prop}
\begin{proof} Follows easily from the fact that if $P\in F[X]$ a monic polynomial with all roots in $F^{\ast}$, then
$$P=\prod_{\lambda\in F^{\ast}}(X-\lambda)^{n_{\lambda}}\quad\text{if and only if}\quad L(P)=\sum_{\lambda\in F^{\ast}}\langle\lambda\rangle_{n_{\lambda}}.$$
\end{proof}
Let us first recall the following definition
\begin{definition} let $(G,\star)$ be a monoid. A semiring $R$ is said to be $G$-graded (or graded by $G$) if there exists a family
$\{R_{\alpha}\}_{\alpha \in G}$ of additive subsemigroups of $R$ satisfying the following conditions:
\begin{enumerate}[label=(\arabic*)]
\item $R=\sum_{\alpha \in G}R_{\alpha}$;
\item $R_{\alpha}R_{\beta}\subseteq R_{\alpha\star\beta}$ for each $\alpha,\beta\in G$;
\item If $a_{\alpha_{1}},\ldots,a_{\alpha_{m}}$ are nonzero elements in $R$ where $\alpha_{i}\neq \alpha_{j}$ if $i\neq j$, and each
$a_{\alpha_{i}}\in R_{\alpha_{i}}$, then $\sum_{i=1}^{m}a_{\alpha_{i}}\neq 0$.
\end{enumerate}
\end{definition}
Let us consider $(\mathcal{L}_{F}^{\ast},+,.)$ where multiplication is given by
$$\langle\lambda\rangle_{i}\langle\mu\rangle_{j}=\langle\lambda\mu\rangle_{i\wedge j},$$
extended via distributivity. Then we have the following theorem.
\begin{thrm}\label{thmm 10} Let $F$ be a field. Then
\begin{enumerate}[label=(\roman*),resume]
\item $(\mathcal{L}_{F}^{\ast},+,.)$ is a commutative semiring graded by the multiplicative monoid $(F^{\ast},.)$.
\item $(\mathcal{L}_{F},+,.)$ is a commutative semiring graded by the multiplicative monoid $(F,.)$.
\end{enumerate}
\end{thrm}
\begin{proof}
\begin{enumerate}[label=(\roman*)]
\item The fact that $(\mathcal{L}_{F}^{\ast},+,.)$ is a commutative semiring follows from Theorem~\ref{thm 221}.
In this semiring, $\langle1\rangle_{1}=\Span\{\pmb{1}\}$ is the multiplicative identity and the neutral element
$\langle1\rangle_{0}=\Span\{\pmb{0}\}$ is an absorbing element. On the other hand, we have
    \begin{equation*}
    \mathcal{L}_{F}^{\ast}=\Sigma_{\lambda\in F^{\ast}}I_{\lambda}
    \end{equation*}
     where $\{I_{\lambda}\}_{\lambda\in F^{\ast}}$ is a family of additive subsemigroups of $\mathcal{L}_{F}^{\ast}$ satisfying
     $I_{\lambda}I_{\mu}\subseteq I_{\lambda\mu}$ by Theorem~\ref{thm 221}. Let now
     $\langle\lambda_{1}\rangle_{t_{1}},\ldots,\langle\lambda_{m}\rangle_{t_{m}}$ be nonzero elements of $\mathcal{L}_{F}^{\ast}$. Since
     \begin{equation*}
     \langle\lambda_{i}\rangle_{t_{i}}\subseteq \langle\lambda_{1}\rangle_{t_{1}}+\cdots+\langle\lambda_{m}\rangle_{t_{m}},
     \end{equation*}
      it follows that
       \begin{equation*}
      \langle\lambda_{1}\rangle_{t_{1}}+\cdots+\langle\lambda_{m}\rangle_{t_{m}}\neq \Span\{\pmb{0}\}.
      \end{equation*}
       This completes the proof of Part $(i)$.
\item Follows immediately from $(i)$ and the fact that $$\langle0\rangle_{s}\langle\lambda\rangle_{t}=\langle0\rangle_{s\vee_{\lambda}t}.$$
\end{enumerate}
\end{proof}
Now we are going to show that $(\mathcal{L}_{F},+,\ast)$ is also a commutative semiring graded by the multiplicative monoid $(F,.)$.
\begin{lem}\label{lem01m} Let $s$ and $t$ be any positive integers and $m$ be any nonnegative integer. Then for all nonzero elements $\alpha, \lambda$ of $F$, we have
\begin{enumerate}[label=(\arabic*)]
\item $\langle 0\rangle_{0}\ast\langle 0\rangle_{m}=\langle 0\rangle_{0}=\langle 0\rangle_{0\wedge m}$.
\item $\langle 0\rangle_{s}\ast\langle 0\rangle_{t}=\langle 0\rangle_{s\wedge t}$.
\item $\langle 0\rangle_{s}\ast\langle \lambda\rangle_{1}=\langle \lambda\rangle_{s}=\langle \lambda\rangle_{s\wedge 1}$.
\item $\langle \alpha\rangle_{1}\ast\langle \lambda\rangle_{1}=\langle \alpha+\lambda\rangle_{1}$.
\item $\langle \alpha\rangle_{s}\ast\langle \lambda\rangle_{t}=\langle \alpha+\lambda\rangle_{s\wedge t}$.
\end{enumerate}
In other words, we have for all $e,f\in \mathbb{N}$ and for all $\gamma,\beta\in F$:
$$\langle \gamma\rangle_{e}\ast\langle \beta\rangle_{f}=\langle \gamma+\beta\rangle_{e\wedge f}.$$
\end{lem}
\begin{proof}
\begin{enumerate}[label=(\arabic*)]
\item Trivial.
\item We have $\pmb{0}_{e}\ast \pmb{0}_{f}=\pmb{\alpha}$ where
$$\alpha_{n}=\sum_{i=0}^{n}\binom{n}{i}\pmb{0}_{e}(i)\pmb{0}_{f}(n-i)=\binom{e+f}{e}\pmb{0}_{e+f}(n);$$
in other words, $\pmb{0}_{e}\ast \pmb{0}_{f}=\binom{e+f}{e}\pmb{0}_{e+f}$. This shows that
$$\langle 0\rangle_{s}\ast\langle 0\rangle_{t}=\Span\{\binom{e+f}{e}\pmb{0}_{e+f}/(e,f)\in\nabla(s-1,t-1)\}.$$
Hence we deduce that the inclusion $\langle 0\rangle_{s}\ast\langle 0\rangle_{t}\subseteq \langle 0\rangle_{s\wedge t}$ holds.
\\The other inclusion is an immediate consequence of Lemma~\ref{lem 0110}.
\item We have $\pmb{0}_{e}\ast \pmb{\lambda}=\pmb{\alpha}$ where
$$\alpha_{n}=\sum_{i=0}^{n}\binom{n}{i}\pmb{0}_{e}(i)\lambda^{n-i}=\binom{n}{e}\lambda^{n-e}.$$ This means that
$\pmb{0}_{e}\ast \pmb{\lambda}=\lambda^{-e}\pmb{\lambda}\Lambda_{e}$, and as consequence
$\langle 0\rangle_{s}\ast\langle \lambda\rangle_{1}=\langle \lambda\rangle_{s}$.
\item Follows from the fact that $\pmb{\alpha}\ast \pmb{\lambda}=\pmb{\beta}$ where
$$\beta_{n}=\sum_{i=0}^{n}\binom{n}{i}\alpha^{i}\beta^{n-i}=(\alpha+\beta)^{n}.$$
\item We have
\begin{eqnarray*}
\langle \alpha\rangle_{s}\ast\langle \lambda\rangle_{t}&=
(\langle 0\rangle_{s}\ast\langle \alpha\rangle_{1})\ast(\langle 0\rangle_{t}\ast\langle \lambda\rangle_{1})\\
&=\langle 0\rangle_{s}\ast\langle 0\rangle_{t}\ast \langle \alpha\rangle_{1}\ast\langle \lambda\rangle_{1}\\
&=\langle 0\rangle_{s\wedge t}\ast\langle \alpha+\lambda\rangle_{1}\\
&=\langle \alpha+\lambda\rangle_{s\wedge t}.
\end{eqnarray*}
\end{enumerate}
\end{proof}
\begin{thrm}\label{thmm 10} Let $F$ be a field. Then
$(\mathcal{L}_{F},+,\ast)$ is a commutative semiring graded by the multiplicative monoid $(F,.)$.
\end{thrm}
\begin{proof}
Follows from Lemma~\ref{lem01m}.
\\We note that in this semiring, $\langle0\rangle_{1}=\Span\{\pmb{1^{\ast}}\}$ is the multiplicative identity
and the neutral element $\langle1\rangle_{0}=\Span\{\pmb{0}\}$ is an absorbing element.
\end{proof}
\section{Applications}
Let us partition
\begin{equation*}
{F^{\star}}^{m}=\bigcup_{i\in T_{1}} \Omega_{i}
\end{equation*}
and
\begin{equation*}
F^{m}=\bigcup_{i\in T_{2}} \Phi_{i}
\end{equation*}
into equivalence classes under the equivalence relations $\mathcal{R}_{1}$ and $\mathcal{R}_{2}$ respectively, where
$$(\lambda_{1},\ldots,\lambda_{m})\mathcal{R}_{1}(\mu_{1},\ldots,\mu_{m}) \quad\text{if and only if}\quad \lambda_{1}\cdots\lambda_{m}=\mu_{1}\cdots\mu_{m}.$$
$$(\lambda_{1},\ldots,\lambda_{m})\mathcal{R}_{2}(\mu_{1},\ldots,\mu_{m}) \quad\text{if and only if}\quad \lambda_{1}+\cdots+\lambda_{m}=\mu_{1}+\cdots+\mu_{m}.$$
Let $\wedge_{m}$ denote the map defined as follows
\begin{align*}
\wedge_{m}: \mathbb{N}^{m}&\longrightarrow \mathbb{N} \\
(t_{1},\ldots,t_{m})&\longmapsto t_{1}\wedge\cdots\wedge t_{m}.
\end{align*}
Let $\xi_{m}: F^{m}\longrightarrow \mathbb{N}^{m}$ be a mapping satisfying Condition $(\mathcal{D})$ below.
\begin{eqnarray*}
(\mathcal{D}):
&\text{For all}\,\,\sigma\in S_{m},\,\,\text{the symmetric group of degree}\,\,m,\\
&\xi_{m}(\mu_{\sigma(1)},\ldots,\mu_{\sigma(m)})=(t_{\sigma(1)},\ldots,t_{\sigma(m)})\\
&\text{whenever}\,\,(t_{1},\ldots,t_{m})=\xi_{m}(\mu_{1},\ldots,\mu_{m}).
\end{eqnarray*}
Suppose further that $\xi_{m}$ are with finite supports, and let us use the following notations for simplicity
\renewcommand{\arraystretch}{1.5}
$$\begin{array}{llll}
\overline{\xi_{m}}&=&\wedge_{m}\circ \xi_{m}&\\
\Omega_{i}^{\xi_{m}}&=&\max_{(\mu_{1},\ldots,\mu_{m})\in \Omega_{i}}(\overline{\xi_{m}}(\mu_{1},\ldots,\mu_{m}))&\\
\Phi_{i}^{\xi_{m}}&=&\max_{(\mu_{1},\ldots,\mu_{m})\in \Phi_{i}}(\overline{\xi_{m}}(\mu_{1},\ldots,\mu_{m}))&\\
\widehat{\Omega}_{i}&=&\mu_{1}\cdots\mu_{m}&\hspace*{-5.5pc}\text{if}\quad (\mu_{1},\ldots,\mu_{m})\in \Omega_{i}\\
\widehat{\Phi}_{i}&=&\mu_{1}+\cdots+\mu_{m}&\hspace*{-5.5pc}\text{if}\quad (\mu_{1},\ldots,\mu_{m})\in \Phi_{i}\\
\langle\mu_{1}\rangle_{\xi_{m}}\cdots\langle\mu_{m}\rangle_{\xi_{m}}&=&\langle\mu_{1}\cdots\mu_{m}\rangle_{\overline{\xi_{m}}(\mu_{1},\ldots,\mu_{m})}&
\hspace*{-5.5pc}\text{if}\quad (\mu_{1},\ldots,\mu_{m})\in {F^{\star}}^{m}\\
\langle\mu_{1}\rangle_{\xi_{m}}\odot\cdots\odot\langle\mu_{m}\rangle_{\xi_{m}}&=&
\langle\mu_{1}+\cdots+\mu_{m}\rangle_{\overline{\xi_{m}}(\mu_{1},\ldots,\mu_{m})}&
\end{array}$$
From Theorem~\ref{thm 221} and Lemma~\ref{lem01m} we have
$$\langle\mu_{1}\rangle_{\xi_{m}}\cdots\langle\mu_{m}\rangle_{\xi_{m}}=
\langle\mu_{1}\rangle_{t_{1}}\cdots\langle\mu_{m}\rangle_{t_{m}}\quad \forall(\mu_{1},\ldots,\mu_{m})\in {F^{\star}}^{m}$$
and
$$\langle\mu_{1}\rangle_{\xi_{m}}\odot\cdots\odot\langle\mu_{m}\rangle_{\xi_{m}}=
\langle\mu_{1}\rangle_{t_{1}}\ast\cdots\ast\langle\mu_{m}\rangle_{t_{m}}\quad \forall(\mu_{1},\ldots,\mu_{m})\in F^{m}$$
 whenever $\xi_{m}(\mu_{1},\ldots,\mu_{m})=(t_{1},\ldots,t_{m})$.
\\Under these notations, we have the following lemma:
\begin{lem}\label{lem 245} Let $m$ be a positive integer, $\xi_{m}: F^{m}\longrightarrow \mathbb{N}^{m}$
be a mappping with finite support satisfying Condition~$(\mathcal{D})$. Then we have
\begin{enumerate}[label=(\roman*)]
\item $\sum_{(\lambda_{1},\ldots,\lambda_{m})\in {F^{\star}}^{m}}\langle\lambda_{1}\rangle_{\xi_{m}}\cdots\langle\lambda_{m}\rangle_{\xi_{m}}=
\sum_{i\in T_{1}}\langle\widehat{\Omega}_{i}\rangle_{\Omega_{i}^{\xi_{m}}}.$
\item $\sum_{(\lambda_{1},\ldots,\lambda_{m})\in F^{m}}\langle\lambda_{1}\rangle_{\xi_{m}}\odot\cdots\odot\langle\lambda_{m}\rangle_{\xi_{m}}=
\sum_{i\in T_{2}}\langle\widehat{\Phi}_{i}\rangle_{\Phi_{i}^{\xi_{m}}}.$
\end{enumerate}
\end{lem}
\begin{proof}
\begin{enumerate}[label=(\roman*)]
\item
We have \begin{align*}
\sum_{(\lambda_{1},\ldots,\lambda_{m})\in {F^{\star}}^{m}}\langle\lambda_{1}\rangle_{\xi_{m}}\cdots\langle\lambda_{m}\rangle_{\xi_{m}}&=
\sum_{i\in T_{1}}\sum_{(\lambda_{1},\ldots,\lambda_{m})\in \Omega_{i}}\langle\lambda_{1}\rangle_{\xi_{m}}\cdots\langle\lambda_{m}\rangle_{\xi_{m}}\\
&=\sum_{i\in T_{1}}\sum_{(\lambda_{1},\ldots,\lambda_{m})\in \Omega_{i}}
\langle\lambda_{1}\cdots\lambda_{m}\rangle_{\overline{\xi_{m}}(\lambda_{1},\ldots,\lambda_{m})}\\
&=\sum_{i\in T_{1}}\sum_{(\lambda_{1},\ldots,\lambda_{m})\in \Omega_{i}}
\langle\widehat{\Omega}_{i}\rangle_{\overline{\xi_{m}}(\lambda_{1},\ldots,\lambda_{m})}\\
&=\sum_{i\in T_{1}}\langle\widehat{\Omega}_{i}\rangle_{\Omega_{i}^{\xi_{m}}}
\end{align*}
\item The proof of this is analogous to that of $(i)$.
\end{enumerate}
\end{proof}
\begin{remark}
In Lemma~\ref{lem 245}, we may replace the condition that $\xi_{m}$ (resp. $\xi_{m}$) is with finite support
with the more general assumption that $\overline{\xi_{m}}$ (resp. $\overline{\xi_{m}}$) is bounded on every $\Omega_{i}$ (resp. $\Phi_{i}$).
\end{remark}
\begin{remark}
Note that Condition~$(\mathcal{D})$ is a sufficient condition to guarantee that for all $\sigma\in S_{m}$,
$$\langle\mu_{\sigma(1)}\rangle_{\xi_{m}}\cdots\langle\mu_{\sigma(m)}\rangle_{\xi_{m}}=
\langle\mu_{1}\rangle_{\xi_{m}}\cdots\langle\mu_{m}\rangle_{\xi_{m}}\quad \forall(\mu_{1},\ldots,\mu_{m})\in {F^{\star}}^{m}$$
and
$$\langle\mu_{\sigma(1)}\rangle_{\xi_{m}}\odot\cdots\odot\langle\mu_{\sigma(m)}\rangle_{\xi_{m}}=
\langle\mu_{1}\rangle_{\xi_{m}}\odot\cdots\odot\langle\mu_{m}\rangle_{\xi_{m}}.$$
\end{remark}
We may now state the following result.
\begin{thrm}\label{thmrmrm} Let $F$ be a field and $\overline{F}$ be an algebraic closure of $F$. Let
$P_{i}=\prod_{\lambda\in \overline{F}^{\ast}}(X-\lambda)^{\lambda(P_{i})}$, $1\leq i\leq m$, be monic
polynomials with all roots in $\overline{F}^{\ast}$, where
    $\lambda(P_{i})$ designates the multiplicity of $\lambda$ in $P_{i}$.
    Let $\xi_{m}$ be the map defined by
\begin{align*}
\xi_{m}: \overline{F}^{m}&\longrightarrow \mathbb{N}^{m}\\
(\lambda_{1},\ldots,\lambda_{m})&\longmapsto (\lambda_{1}(P_{1}),\ldots,\lambda_{m}(P_{m})).
\end{align*} Then
\begin{enumerate}[label=(\roman*)]
\item $\prod_{i=1}^{m}L_{\overline{F}}(P_{i})=L_{\overline{F}}(\Upsilon(P_{1},\ldots,P_{m})),$
where
\begin{equation*}
\Upsilon(P_{1},\ldots,P_{m})=\prod_{i\in T_{1}}(X-\widehat{\Omega}_{i})^{\Omega_{i}^{\xi_{m}}}
\end{equation*}
\item $\Upsilon(P_{1},\ldots,P_{m})=\prod_{i\in T_{1}}(X-\widehat{\Omega}_{i})^{\Omega_{i}^{\xi_{m}}}$ is a polynomial over $F$ and
$$\prod_{i=1}^{m}L_{F}(P_{i})=L_{F}(\Upsilon(P_{1},\ldots,P_{m})).$$
\end{enumerate}
\end{thrm}
\begin{proof}
\begin{enumerate}[label=(\roman*)]
\item
Let
\begin{equation*}
\Upsilon(P_{1},\ldots,P_{m})=\prod_{i\in T_{1}}(X-\widehat{\Omega}_{i})^{\Omega_{i}^{\xi_{m}}}.
 \end{equation*}
 Since the $\widehat{\Omega}_{i}$'s are pairwise distinct elements of $\overline{F}^{\ast}$, it follows that
 \begin{equation*}
 L_{\overline{F}}(\Upsilon(P_{1},\ldots,P_{m}))=\sum_{i\in T_{1}}\langle\widehat{\Omega}_{i}\rangle_{\Omega_{i}^{\xi_{m}}}.
 \end{equation*}
  On the other hand, by Lemma~\ref{lem 245}, we have \begin{align*}
\prod_{i=1}^{m}L_{\overline{F}}(P_{i})&=\prod_{i=1}^{m}(\sum_{\lambda\in \overline{F}^{\ast}}\langle\lambda\rangle_{\lambda(P_{i})})\\
&=\sum_{(\lambda_{1},\ldots,\lambda_{m})\in {\overline{F}^{\star}}^{m}}\langle\lambda_{1}\rangle_{\xi_{m}}\cdots\langle\lambda_{m}\rangle_{\xi_{m}}\\
&=\sum_{i\in T_{1}}\langle\widehat{\Omega}_{i}\rangle_{\Omega_{i}^{\xi_{m}}}.
\end{align*}
Accordingly,
\begin{equation*}
\prod_{i=1}^{m}L_{\overline{F}}(P_{i})=L_{\overline{F}}(\Upsilon(P_{1},\ldots,P_{m})).
\end{equation*}
\item Follows immediately from Lemma~2.0 of~\cite{Chin} and the easily shown fact that
\begin{equation*}
L_{F}(P_{1})\cdots L_{F}(P_{m})=\mathcal{C}_{F}\bigcap L_{\overline{F}}(P_{1})\cdots L_{\overline{F}}(P_{m}).
\end{equation*}
\end{enumerate}
\end{proof}
\begin{remark} We note that if one of the $P_{i}$ is $1$, then $\Omega_{i}^{\xi_{m}}=0$ for all $i\in T_{1}$; that is $\Upsilon(P_{1},\ldots,P_{m})=1$.
\end{remark}
Now let $P_{i}\in F[X]$, $1\leq i\leq m$, be nonconstant monic polynomials. Then
\begin{equation*}
L_{\overline{F}}(P_{i})=\langle0\rangle_{s_{i}}+L_{\overline{F}}(Q_{i}),\,\, 1\leq i\leq m,
 \end{equation*}
 where $s_{i}=0(P_{i})$ is the multiplicity of $0$ as a root of $P_{i}$ and $0(Q_{i})=0, 1\leq i\leq m.$
\\Since
 $$
 \langle0\rangle_{t}L_{\overline{F}}(Q)=
 \langle0\rangle_{t}\sum_{\lambda\in \overline{F}^{\ast}}\langle\lambda\rangle_{\lambda(Q)}$$
 for all polynomial $Q\neq1$ such that $Q(0)\neq0$, it follows from Lemma~\ref{lemzm} that
   \begin{equation*}
   \langle0\rangle_{t}L_{\overline{F}}(Q)=\langle0\rangle_{t}\,\,\text{if}\,\, L_{\overline{F}}(Q)\neq\Span\{\pmb{0}\},
   \end{equation*}
    i.e. if $Q(0)\neq0$ and $Q\neq 1$. Thus it is naturel to consider the following set
    \begin{equation*}
    \Theta=\{i/Q_{i}= 1\}.
     \end{equation*}
     Let
    \begin{eqnarray*}
\rho=\begin{cases}
\min\{s_{i}/i\in \Theta\}&\quad\text{if}\quad \Theta\neq\emptyset\\
\max\{s_{i}/1\leq i\leq m\}& \quad\text{otherwise}.
\end{cases}
\end{eqnarray*}
Note that since we have assumed that $P_{i}\neq1$, $1\leq i\leq m$, we have that $s_{i}\neq 0$ whenever
$i\in \Theta$. Using this remark it is straightforward to check that
$$
\prod_{i=1}^{m}L_{\overline{F}}(P_{i})=
\left\{\begin{array}{lll}
\langle0\rangle_{\rho}+ L_{\overline{F}}(\Upsilon(Q_{1},\ldots,Q_{m})) &\quad\text{if}\quad &\Theta =\emptyset\\
\langle0\rangle_{\rho} &\quad\text{if}\quad &\Theta \neq \emptyset
\end{array}\right.$$
and since $L_{\overline{F}}(\Upsilon(Q_{1},\ldots,Q_{m}))=0$ whenever one of the $Q_{i}$ is $1$, it follows that, even if $\Theta \neq \emptyset$, we have
$$\prod_{i=1}^{m}L_{\overline{F}}(P_{i})=\langle0\rangle_{\rho}+ L_{\overline{F}}(\Upsilon(Q_{1},\ldots,Q_{m})).$$
Since none of the roots of $\Upsilon(Q_{1},\ldots,Q_{m})$ is equal to zero, we have
\begin{eqnarray*}
\prod_{i=1}^{m}L_{\overline{F}}(P_{i})&=& \langle0\rangle_{\rho}+ L_{\overline{F}}(\Upsilon(Q_{1},\ldots,Q_{m}))\\
&=&L_{\overline{F}}(X^{\rho}\Upsilon(Q_{1},\ldots,Q_{m})),
\end{eqnarray*}
and consequently,
\begin{eqnarray*}
\prod_{i=1}^{m}L_{F}(P_{i})=L_{F}(X^{\rho}\Upsilon(Q_{1},\ldots,Q_{m})).
\end{eqnarray*}
Thus we have proved the following:
\begin{thrm}\label{Them 11} Let $P_{1}=X^{s_{1}}Q_{1},\ldots,P_{m}=X^{s_{m}}Q_{m}$ be nonconstant monic polynomials over $F$,
where $Q_{i}(0)\neq 0,\, 1\leq i\leq m$.
Then
 $$\prod_{i=1}^{m}L(P_{i})=L(X^{\rho}\Upsilon(Q_{1},\ldots,Q_{m})),$$
where
\begin{eqnarray*}
\rho=\begin{cases}
\min\{s_{i}/i\in \Theta\}&\quad\text{if}\quad \Theta\neq\emptyset\\
\max\{s_{i}/1\leq i\leq m\}& \quad\text{otherwise},
\end{cases}
\end{eqnarray*}
and $\Theta=\{i/Q_{i}= 1\}$.
\end{thrm}
The following result is the analogous of the theorem above for the Hurwitz product.
\begin{thrm}\label{thrrm} Let $P_{1},\ldots,P_{m}$ be nonconstant monic polynomials over $F$. Let $\Psi(P_{1},\ldots,P_{m})$ be the polynomial
obtained using the map
\begin{align*}
\xi_{m}: \overline{F}^{m}&\longrightarrow \mathbb{N}^{m}\\
(\lambda_{1},\ldots,\lambda_{m})&\longmapsto (\lambda_{1}(P_{1}),\ldots,\lambda_{m}(P_{m})).
 \end{align*}
 Then we have
 \begin{enumerate}[label=(\arabic*)]
 \item
\begin{equation*}
\Psi(P_{1},\ldots,P_{m})=\prod_{i\in T_{2}}(X-\widehat{\Phi}_{i})^{\Phi_{i}^{\xi_{m}}}
\end{equation*} is a polynomial over $F$.
\item $$L(P_{1})\ast\cdots\ast L(P_{m})=L(\Psi(P_{1},\ldots,P_{m})),$$
\end{enumerate}
\end{thrm}
\begin{proof}
One can derive the above result in similar manner as given for Theorem~\ref{thmrmrm} and~\ref{Them 11}. We omit the details.
\end{proof}

We end this section by the following remark, in which we point out the incorrectness of the main result, Theorem~$2.3$, of~\cite{Chin}, and we
identify mistakes that led to this incorrectness.
\begin{remark}\label{remer}~
\begin{enumerate}[label=(\arabic*)]
\item We point out that Lemma~$2.2$ of~\cite{Chin}, which is key in the proof of the main theorem in~\cite{Chin}, is not correct. Recall that this lemma says
\begin{lem}\label{lem1m}
Suppose that $t=t(i,j)< n\leq i+j$, then
\begin{enumerate}[label=(\roman*)]
\item  $\binom{t}{i}\neq0$,
\item $\binom{n}{i}=0$.
\end{enumerate}
\end{lem}
To check that this result is incorrect it is sufficient to take $F$ any field with characteristic $3$, $i=1$, $j=3$ and $n=4$. In this case,
$t=t(1,3)=3$, $\binom{t}{i}=0$ and $\binom{n}{i}\neq0$.
\item The Corollary~\ref{cor1r} can be regarded as a correct form of Lemma~\ref{lem1m}.
\item The assertion $(i)$ of Proposition~$2.3$ of~\cite{Chin} is not correct. The reason is the following:
\begin{enumerate}[label=(\alph*)]
\item
From the second line of the proof Proposition~$2.3$ of~\cite{Chin}, we have
$$d_{r}.d_{s}=\sum_{n=\max(r,s)}^{t}\binom{n}{s}\binom{s}{n-r}d_{n},$$
which is not true. The corrected identity is
$$d_{r}.d_{s}=\sum_{n=\max(r,s)}^{d(r,s)}\binom{n}{s}\binom{s}{n-r}d_{n},$$ where $d(r,s)=(r+1)\wedge(s+1)-1$.
\item From the line~$5$ and $6$, we have
$$d_{r}.d_{t-r}=\binom{t}{r}d_{t}+\ldots\,.\,\text{By Lemma~2.2(ii)},\,\binom{t}{r}\not\equiv0,$$
$$\text{and we have}\, d_{t}\in D_{r}.D_{s}.$$
This are two reasons this is not true. The first reason is that $\binom{t}{r}$ may be congruent to $0$ modulo $p=\text\rm{char}(F)>0$, as is noted
in~$(1)$ above.
\\The second reason is that even if this identity is correct, no result in the paper~\cite{Chin} guarantees that
$$\sum_{n=j}^{t-1}\binom{n}{t-r}\binom{t-r}{n-r}d_{n}=d_{r}.d_{t-r}-\binom{t}{r}d_{t},\,\, j=\max(r,t-r)$$
is an element of $D_{r}.D_{s}$.
\end{enumerate}
\item The following simple counter-example demonstrates that the main result in Chin and Goldman's paper is incorrect.
\\Let $F$ be any field with characteristic $3$, and let $i=1$, $j=3$. Then $1\wedge 3=3$ and $t(1-1,3-1)=t(0,2)=1$.
\\Using Chin and Goldman's theorem one obtain the following two results
$$L(x-1)L((X-1)^{3})=L((X-1)^{t(0,2)})=L(X-1)$$
and
$$L(x-1)\ast L((X-1)^{3})=L((X-2)^{t(0,2)})=L(X-2),$$
which clearly contradict the facts that
$$\pmb{\Lambda}_{0}\pmb{\Lambda}_{2}=\pmb{\Lambda}_{2}\not\in L(X-1)$$
and
$$\pmb{\Lambda}_{0}\ast\pmb{\Lambda}_{2}=2^{-2}\pmb{2}\pmb{\Lambda}_{2}=\pmb{2}\pmb{\Lambda}_{2}\not\in L(X-2).$$
\item Finally, we note that by comparing our main results with that of~\cite{Chin}, we see that the main result
of the aforementioned paper will become correct if we redefine the function $t$ as follows:
\begin{eqnarray*}
t(i,j)=\begin{cases}
i+j-1&\,\,\text{if}\,\, i=1 \,\,\text{or}\,\, j=1\\
(i+1)\wedge (j+1)& \,\,\text{otherwise}
\end{cases}
\end{eqnarray*}
and in the formula given $t_{0}$, the multiplicity of $0$ as a root of $h$, we replace $r(0)-1$ and $s(0)-1$ by $r(0)$ and $s(0)$.
\end{enumerate}
\end{remark}
\section{Some algebraic properties}
Let us begin this section with the following definition.
\begin{definition}
A mapping $f$ from semiring $(S,+,.)$ to semiring $(T,\oplus, \centerdot)$ is said to be a semiring homomorphism if for any $a,b\in S$
\begin{itemize}
\item $f(a+b)=f(a)\oplus f(b)$,
\item $f(a.b)=f(a)\centerdot f(b)$,
\item $f(0_{S})=0_{T}$ and $f(1_{S})=1_{T}$.
\end{itemize}
If moreover $f$ is bijective, then $f$ is said to be an isomorphism.
\end{definition}
In the following result we determine the invertible and the idempotent elements of $(\mathcal{L}_{F},+,.)$ and $(\mathcal{L}_{F},+,\ast)$. Before that we need to give the following lemma.
\begin{lem}\label{lem 197} Let $i,j$ be positive integers and $s,t$ nonnegative integers. Let $p=\text\rm{char}(F)$. Then
\begin{enumerate}[label=(\arabic*)]
\item Suppose $p>0$, then
\begin{enumerate}[label=(\roman*))]
\item $i\wedge j=i+j-1$ if and only if $q(i,j)=0$, where $q(i,j)$ is the integer defined in formula~\ref{for 21}.
\item $ip^{s}\wedge jp^{s}=p^{s}(i\wedge j)$.
\item $ip^{s}\wedge j$ is divisible by $p^{s}$.
\item $ip^{s}\wedge jp^{t}$ is divisible by $p^{s}\vee p^{t}$.
\end{enumerate}
\item \begin{enumerate}[label=(\roman*)]
\item $i\wedge j=1$ if and only if $i=j=1$.
\item Suppose that $p=0$. Then $i\wedge i=i$ if and only if $i=1$.
\item Suppose $p>0$ and $\textrm{gdc}(i,p)=1$. Then $i\wedge i=i$ if and only if $i=1$.
\item Suppose $p>0$. Then $i\wedge i=i$ if and only if $i=p^{n}$.
\end{enumerate}
\end{enumerate}
\end{lem}
\begin{proof}~
\begin{enumerate}[label=(\arabic*)]
\item\begin{enumerate}[label=(\roman*)]
\item is straightforward to check.
 \end{enumerate}
\noindent To prove $(4)$ and the remainder parts of $(3)$, let us consider the $p$-adic digit expansions of $i-1$ and $j-1$
$$
i-1=\sum_{m\geq 0}i_{m}p^{m}\quad\text{and}\quad
  j-1=\sum_{m\geq 0}j_{m}p^{m}.$$
\begin{enumerate}[label=(\roman*),resume]
\item Let $q=q(i,j)$ be the smallest nonnegative integer such that $i_{m}+j_{m}<p$ for all $m\geq q$. Since
\begin{equation*}
ip^{s}-1=(i-1)p^{s}+p^{s},
 \end{equation*}
 it follows that
\begin{eqnarray*}
 ip^{s}-1&=&\sum_{m\geq 0}(ip^{s})_{m}\\
 &=&p-1+(p-1)p+\cdots+(p-1)p^{s-1}+\sum_{m\geq 0}i_{m}p^{m+s}
 \end{eqnarray*}
  is the $p$-adic digit expansion of $ip^{s}-1$. Similarly,
  \begin{eqnarray*}
  jp^{s}-1&=&\sum_{m\geq 0}(jp^{s})_{m}\\
  &=&p-1+(p-1)p+\cdots+(p-1)p^{s-1}+\sum_{m\geq 0}j_{m}p^{m+s}
  \end{eqnarray*}
   is the $p$-adic digit expansion of $jp^{s}-1$. It is clair that $q+s$ is the smallest nonnegative integer such that
   \begin{equation*}
   (ip^{s})_{m}+(jp^{s})_{m}<p \quad\text{for all}\quad m\geq q+s,
   \end{equation*}
   i.e., $q+s=q(ip^{s},jp^{s})$.
    Thus
   \begin{eqnarray*}
   ip^{s}\wedge jp^{s}&=&p^{s+q}+\sum_{m\geq q}(i_{m}+j_{m})p^{s+m}\\
   &=&p^{s}(i\wedge j).
  \end{eqnarray*}
\item Let $q=q(ip^{s},j)$ be the smallest nonnegative integer such that
\begin{equation*}
(ip^{s})_{m}+j_{m}<p \quad\text{for all}\quad m\geq q.
\end{equation*}
If $q\geq s$, then $ip^{s}\wedge j$, which is divisible by $p^{q}$, is certainly divisible by $p^{s}$. Suppose that $q\leq s-1$. Since
\begin{equation*}
p-1+(p-1)p+\cdots+(p-1)p^{s-1}+\sum_{m\geq 0}i_{m}p^{m+s}
\end{equation*}
 is the $p$-adic digit expansion of $ip^{s}-1$ and
\begin{equation*}
 j_{0}+j_{1}p+\cdots+j_{s-1}p^{s-1}+\sum_{m\geq 0}j_{m+s}p^{m+s}
\end{equation*}
  is that of $j-1$, it follows that
  \begin{equation*}
  j_{q}=\cdots=j_{s-1}=0
  \end{equation*}
   and then
  \begin{eqnarray*}
   ip^{s}\wedge j&=&p^{q}+(p-1)p^{q}+\cdots+(p-1)p^{s-1}+\sum_{m\geq 0}(i_{m}+j_{m+s})p^{m+s}\\
   &=&p^{s}+p^{s}\sum_{m\geq 0}(i_{m}+j_{m+s})p^{m}
  \end{eqnarray*}
    is divisible by $p^{s}$. In both cases $ip^{s}\wedge j$ is divisible by $p^{s}$, as was to be shown.
     \item Follows immediately from $(ii)$ and $(iii)$.
     \end{enumerate}
    \end{enumerate}
\begin{enumerate}[label=(\arabic*),resume]
\item\begin{enumerate}[label=(\roman*)]
\item follows obviously from the fact that $1\leq i\vee j\leq i\wedge j$.
\item is obvious.
\item follows from the fact that $i\wedge i=p^{q(i-1,i-1)}+2\sum_{m\leq q(i-1,i-1)}i_{m}p^{m}$ which is divisible by $p$ if $q(i-1,i-1)\neq0$.
\item It is clear that $p^{n}\wedge p^{n}=p^{n}(1\wedge 1)=p^{n}$. Suppose that $i\wedge i=i$. Put $i=p^{n}s$ where $\textrm{gcd}(s,p)=1$. We have
$i\wedge i=p^{n}(s\wedge s)=p^{n}s$, then $s\wedge s=s$, hence $s=1$ and $i=p^{n}$.
\end{enumerate}
\end{enumerate}
\end{proof}
\begin{remark}
For another proof of Part $(1)(iv)$, see~\cite[p.213]{Gott}.
\end{remark}
\begin{thrm} Let $P$ and $Q$ two non constant monic polynomials over $F$ and let $p=\text\rm{char}(F)$. Then
\begin{enumerate}[label=\arabic*.]
\item $L(P)$ is invertible for the Hadamard product if and only if $P=X-\alpha$ where $\alpha\in F^{\star}$.
\item $L(P)$ is invertible for the Hurwitz product if and only if $P=X-\alpha$ where $\alpha\in F$.
\item If $\text\rm{char}(F)=0$ then:
$$L(P)L(P)=L(P)\quad\text{if and only if}\,\, P=X^{t} \,\,\text{or}\quad P=X^{m}(X^{n}-1)$$
 where $t,n$ are positive integers and $m$ is a nonnegative integer.
    \\If $\text\rm{char}(F)=p>0$ then:
    $$L(P)L(P)=L(P)\,\,\text{if and only if}\quad P=X^{t} \,\,\text{or}\quad P=X^{m}(X^{n}-1)^{p^{s}}=X^{m}(X^{np^{s}}-1)$$
 where $t,n$ are positive integers and $m,s$ are nonnegative integers.
\item If $\text\rm{char}(F)=0$ then:
$$L(P)\ast L(P)=L(P)\quad\text{if and only if}\quad P=X.$$
\\If $\text\rm{char}(F)=p>0$ then:
$$L(P)\ast L(P)=L(P)\quad\text{if and only if}\quad P=(X^{p^{s}}+a_{1}X^{p^{s-1}}+\cdots+a_{s}X)$$
where $s$ is a nonnegative integer, i.e., $P$ is a $p$-polynomial.
\end{enumerate}
\end{thrm}
\begin{proof}~
\begin{enumerate}[label=\arabic*.]
\item It is clear that $L(X-\alpha)L(X-\alpha^{-1})=L(X-1)$, i.e, $L(X-\alpha)$ is invertible for the Hadamard product. Suppose that there exists $L(Q)\in \mathcal{L}_{F}$ such that $L(P)L(Q)=L(X-1)$. Fix a root $\alpha$ of $P$ and let $\beta$ be a root of $Q$. Then $\alpha\beta=1$, hence $\alpha\in F^{\star}$ and $Q$ admits only $\alpha^{-1}$ as root. For the same reason, $P$ admits one root. Hence $P=X-\alpha$ where $\alpha\in F^{\star}$.
\item Suppose that there exists $L(Q)\in \mathcal{L}_{F}$ such that $L(P)\ast L(Q)=L(X)$. Fix a root $\alpha$ of $P$ and let $\beta$ be a root of $Q$. Then $\alpha+\beta=0$, i,e., $\beta=-\alpha$. $Q$ admits only $-\alpha$ as root. For the same reason, $P$ admits one root. Hence $P=(X-\alpha)^{n}$ and $Q=(X+\alpha)^{m}$. Now since $L(P)\ast L(Q)=L(X^{n\wedge m})$ it follows that $n\wedge m=1$, and then $n=m=1$. The converse is obvious.
\item Suppose that $L(P)L(P)=L(P)$ and put $P=X^{m}Q$, $Q(1)\neq0$.
Let $G$ be the set of all non-zero roots of $P$. Suppose that $G\neq\emptyset$. It is clear, from Theorem~\ref{them 11}, that $G$ is closed under multiplication. Let $\alpha\in G$. Then $\alpha,\alpha^{2},\alpha^{3},\ldots$ are all in $G$. Since $G$ is finite, there exists a positive integer $n$ such that $\alpha^{n}=1$, in particular $\alpha^{-1}=\alpha^{n-1}\in G$. Hence $G$ is a finite subgroup of the multiplicative group of $\overline{F}$. then $G$ is cyclic, i.e., there is $\alpha\in G$ such that $G=\{1,\alpha,\ldots,\alpha^{n-1}\}$ where $n$ is the order of $G$. Consequently, $Q=\prod_{q=1}^{n}(X-\alpha^{q-1})^{m_{q}}$.
Now since $L(P)L(P)=L(P)$, it follows that $L(\Upsilon(Q,Q))=L(Q)$. Let $i\in S_{1}$ such that $\widehat{\Omega}_{i}=1$. Then $\Omega_{i}=\{(1,1),(\alpha,\alpha^{n-1}),\ldots,(\alpha^{n-1},\alpha)\}$. Hence $\widehat{\Omega}_{i}^{\xi_{2}}=\max\{m_{1}\wedge m_{1},m_{2}\wedge m_{n},\ldots,m_{n}\wedge m_{2}\}=m_{1}$. Since $m_{1}\leq m_{1}\wedge m_{1}$, it follows that $m_{1}\wedge m_{1}=m_{1}$ and that $m_{q}\wedge m_{n-q+2}\leq m_{1}, q=2,\ldots,n$. By the same way we show that $m_{1}\wedge m_{q}=m_{q}, q=2,\ldots,n$.
\\If $\text\rm{char}(F)=0$: we have $m_{1}=1$ and then $m_{2}=\cdots=m_{n}=1$. Therefore, $P=X^{m}\prod_{q=1}^{n}(X-\alpha^{q-1})=X^{m}(X^{n}-1)$. It is obvious that this polynomial satisfies $L(P)L(P)=L(P)$.
\\If $\text\rm{char}(F)=p>0$: since $m_{1}\wedge m_{1}=m_{1}$, we have $m_{1}=p^{s}$, and since $m_{1}\wedge m_{q}=m_{q}$, we have $p^{s}\wedge m_{q}=m_{q}$. Then $m_{q}=tp^{s}$. But since $m_{q}\leq m_{q}\wedge m_{n-q+2}\leq m_{1}$, it follows that $t=1$. Therefore, $P=X^{m}\prod_{q=1}^{n}(X-\alpha^{q-1})^{p^{s}}=X^{m}(X^{n}-1)^{p^{s}}$. A straightforward verification proves that this polynomial satisfies $L(P)L(P)=L(P)$.
\item Suppose that $L(P)\ast L(P)=L(P)$. Let $G$ be the set of all roots of $P$. Suppose that $G\neq0$. It is clear, from Theorem~\ref{thrrm}, that $G$ is closed under multiplication. Let $\alpha\in G$. Theorem~\ref{thrrm} entails that $\alpha,2\alpha,3\alpha,\ldots$ are all in $G$. This is not possible only if there exists a positive integer $m$ such that $m\alpha=0$. In particular, $m$ has an opposite element in $G$. Hence $G$ is a subgroup of the additive group of $\overline{F}$. Therefore $P$ is a $p$-polynomial (e.g. see Exercise~$3$, pp.~$411$ of \cite{Jac}). A straightforward verification using Theorem~\ref{thrrm} shows that the converse holds.
\end{enumerate}
\end{proof}
Let us now fix an algebraic closure $\overline{F}$ of $F$. Let $\End(F)$ be the set of all filed endomorphisms of $F$ and put $\Inj(F)=\End(F)\backslash\{0\}$. Let $f\in \End(F)$. We know that $f$ can be extended to an endomorphism of $\overline{F}$. Let $\chi_{f} : F[X]\longrightarrow F[X]$ $P=a_{n}X^{n}+\cdots+a_{0}$, $\chi_{f}(P)=f(a_{n})X^{n}+\cdots+f(a_{0})$. Suppose $P=a(X-\alpha_{1})^{n_{1}}\cdots(X-\alpha_{m})^{n_{m}}$. Let $g,h\in \End(\overline{F})$ be extensions of $f$. Then $g(P)=g(a)(X-g(\alpha_{1}))^{n_{1}}\cdots(X-g(\alpha_{m}))^{n_{m}}=\chi_{f}(P)$ and $h(P)=h(a)(X-h(\alpha_{1}))^{n_{1}}\cdots(X-h(\alpha_{m}))^{n_{m}}=\chi_{f}(P)$. Let us not $\chi_{f}$ simply by $f$. Now fix $\widetilde{f}\in \Inj(\overline{F})$ that extends $f$. Consider the map
$$\varphi_{f}:(\mathcal{L}_{F},+,.)\longrightarrow (\mathcal{L}_{F},+,.)$$ defined by $\varphi_{f}(L(P))=L(f(P))$.
\\It is easy to prove that $\varphi_{f}$ is injective, $\varphi_{f\circ g}=\varphi_{f}\circ\varphi_{g}$ and that $f=g$ whenever $\varphi_{f}=\varphi_{g}$.
Moreover, we have the following theorem which proves that $\varphi_{f}$ is a semiring endomrphism.
\\If $f: F\longrightarrow F$ is an endomorphism we also use $f$ to denote the induced endomorphism
 \begin{align*}
F[X]&\longrightarrow F[X]\\
\sum a_{i}X^{i}&\longmapsto \sum f(a_{i})X^{i}.
 \end{align*}
 \begin{lem}\label{lem1012} Let $P$ and $Q$ be non-constant monic polynomials over $F$ with non-zero terms and let $f\in \Inj(F)$. Consider the following mapping
\begin{align*}
\xi_{2}: \overline{F}^{2}&\longrightarrow \mathbb{N}^{2}\\
(\mu,\lambda)&\longmapsto (\mu(P),\lambda(Q)).
 \end{align*}
 and
 \begin{align*}
\xi^{f}_{2}: \overline{F}^{2}&\longrightarrow \mathbb{N}^{2}\\
(\mu,\lambda)&\longmapsto (\mu(f(P)),\lambda(f(Q)).
 \end{align*}
Let $$S=\{\widehat{\Omega}_{i}/\Omega_{i}^{\xi_{2}}\neq0\}=\{\widehat{\Omega}_{i}/\exists\alpha, \beta: \widehat{\Omega}_{i}=\alpha\beta\}$$
and $$S^{f}=\{\widehat{\Omega}_{i}/\Omega_{i}^{\xi^{f}_{2}}\neq0\}=\{\widehat{\Omega}_{i}/\exists\alpha, \beta: \widehat{\Omega}_{i}=\widetilde{f}(\alpha)\widetilde{f}(\beta)\}.$$
Then the map $\widetilde{f}$ induces a bijection from $S$ onto $S^{f}$ satisfying $\widetilde{f}(\Omega_{i})^{\xi^{f}_{2}}=\Omega_{i}^{\xi_{2}}$.
 \end{lem}
 \begin{proof}
Let $\alpha$ and $\beta$ range over the roots of $P$ and $Q$ in $\overline{F}$, with multiplicities $\alpha(P)$ and $\beta(Q)$, respectively.
\\Since $\widetilde{f}$ is injective, $\widetilde{f}(\alpha)$ and $\widetilde{f}(\beta)$ range over the roots of $f(P)$ and $f(Q)$ in $\overline{F}$, with multiplicities $\alpha(P)$ and $\beta(Q)$, respectively. In other word, $\xi_{2}(\alpha,\beta)=\xi^{f}_{2}(\widetilde{f}(\alpha),\widetilde{f}(\beta)).$
\\Using the fact that $\widetilde{f}$ is injective, it is easily seen that $\mathbf{card}(S)=\mathbf{card}(S^{f})$ and that the map $\widehat{\Omega}_{i}\longmapsto \widetilde{f}(\widehat{\Omega}_{i})$ is a bijection from $S$ onto $S^{f}$ satisfying $\widetilde{f}(\Omega_{i})^{\xi^{f}_{2}}=\Omega_{i}^{\xi_{2}}$.
 \end{proof}
\begin{lem}\label{lemr} Under the same hypothesis as Lemma~\ref{lem1012}, we have
 $$f(\Upsilon(P,Q))=\Upsilon(f(P),f(Q)).$$
\end{lem}
\begin{proof}
Put $S_{1}=\{i/\Omega_{i}^{\xi_{2}}\neq 0\}$ and $S_{1}^{f}=\{i/\Omega_{i}^{\xi^{f}_{2}}\neq0\}$.
From Theorem~\ref{Them 11}, we have
\begin{equation*}
\Upsilon(P,Q)=\prod_{i\in S_{1}}(X-\widehat{\Omega}_{i})^{\Omega_{i}^{\xi_{2}}}.
\end{equation*}
and
\begin{equation*}
\Upsilon(f(P),f(Q))=\prod_{i\in S_{1}^{f}}(X-\widehat{\Omega}_{i})^{\Omega_{i}^{\xi^{f}_{2}}}.
\end{equation*}
\begin{eqnarray*}
f(\Upsilon(P,Q))&=&f(\prod_{i\in S_{1}}(X-\widehat{\Omega}_{i})^{\Omega_{i}^{\xi_{2}}})\\
&=&\prod_{i\in S_{1}}(X-\widetilde{f}(\widehat{\Omega}_{i}))^{\Omega_{i}^{\xi_{2}}}\\
&=&\prod_{i\in S_{1}^{f}}(X-\widetilde{f}(\widehat{\Omega}_{i}))^{\widetilde{f}(\Omega_{i})^{\xi^{f}_{2}}}\\
&=&\Upsilon(f(P),f(Q)).
\end{eqnarray*}
\end{proof}
Let $\Inj(\mathcal{L}_{F})$ be the set of all injective semiring endomorphisms of $\mathcal{L}_{F}$. Then we have
\begin{lem}
Let $\psi\in \Inj(\mathcal{L}_{F})$. Then the following two conditions are equivalent:
\begin{enumerate}[label=(\roman*)]
\item $\psi(L((X-1)^{n})=L((X-1)^{n})$ for all $n\in \mathbb{N}$;
\item There exists $f\in \Inj(F)$, such that $\psi=\varphi_{f}$.
\end{enumerate}
\end{lem}
\begin{proof}
Obviously the condition $(ii)$ implies $(i)$. Suppose that the condition $(i)$ holds and let $f\in \Inj(F)$ be the unique map satisfying $\psi(L(X-\alpha))=\varphi_{f}(L(X-\alpha))$ for all $\alpha\in F$. We have $$\psi(L((X-\alpha)^{n}))=\psi(L((X-1)^{n})L(X-\alpha))=\psi(L((X-1)^{n})\varphi(L(X-\alpha))=$$
$$L((X-1)^{n}))L(X-f(\alpha))=L((X-f(\alpha))^{n})=\varphi_{f}(L((X-\alpha)^{n}).$$
Hence $\psi(L(P))=\varphi_{f}(L(P))$ for all non-constant monic polynomial $P$ over $F$. Thus $\psi=\varphi_{f}$.
\end{proof}
\begin{thrm}
Let $\varphi$ defined by
\begin{align*}
\varphi: \Inj(F)&\longrightarrow \Inj(\mathcal{L}_{F})\\
f&\longmapsto \varphi_{f}.
 \end{align*}
Then $\varphi$ is an injective homomorphism of monoids.
\end{thrm}
\begin{proof} Obviously $\varphi$ is injective and $\varphi(\textrm{id}_{F})=\textrm{id}_{\mathcal{L}_{F}}$.
\\Let $P_{1}=X^{n_{1}}Q_{1}\neq1$ and $P_{2}=X^{n_{2}}Q_{2}\neq1$ where $Q_{1}$ and $Q_{2}$ are polynomials with non-zero constant terms.
\begin{eqnarray*}
L(P)+L(Q)&=&L(\prod_{\lambda\in F}(X-\lambda)^{\lambda(P)})+L(\prod_{\lambda\in F}(X-\lambda)^{\lambda(Q)})\\
&=&L(\prod_{\lambda\in F}(X-\lambda)^{\lambda(P)\vee\lambda(Q)}).
\end{eqnarray*}
Then
\begin{eqnarray*}
\varphi_{f}(L(P)+L(Q))&=&L(\prod_{\lambda\in F}(X-\widetilde{f}(\lambda))^{\lambda(P)\vee\lambda(Q)})\\
&=&L(\prod_{\lambda\in F}(X-\widetilde{f}(\lambda))^{\lambda(P)})+L(\prod_{\lambda\in F}(X-\widetilde{f}(\lambda))^{\lambda(Q)})\\
&=&L(f(P))+L(f(Q)).
\end{eqnarray*}
Hence
$$\varphi_{f}(L(P)+L(Q))=\varphi_{f}(L(P))+\varphi_{f}(L(Q)).$$
Using Theorem~\ref{Them 11}, we have
\begin{equation*}
L(P_{1})L(P_{2})=L(X^{\rho}\Upsilon(Q_{1},Q_{2})),
 \end{equation*}
where
\begin{eqnarray*}
\rho=\begin{cases}
\min\{n_{i}/i\in \Theta\}&\quad\text{if}\quad \Theta\neq\emptyset\\
\max\{n_{1}, n_{2}\}& \quad\text{otherwise}.
\end{cases}
\end{eqnarray*}
and $\Theta=\{i/Q_{i}=1\}$. Then, from Lemma~\ref{lemr}, we have
\begin{equation*}
\varphi_{f}(L(P_{1})L(P_{2}))=L(X^{\rho}\Upsilon(f(Q_{1}),f(Q_{2})).
 \end{equation*}
On the other hand, $$\varphi_{f}(L(P_{1}))\varphi_{f}(L(P_{2}))=L(X^{n_{1}}f(Q_{1}))L(X^{n_{2}}f(Q_{2})).$$ But since $\{i/Q_{i}=1\}=\{i/f(Q_{i})=1\}$, it follows that
$$\varphi_{f}(L(P_{1}))\varphi_{f}(L(P_{2}))=L(X^{\rho}\Upsilon(f(Q_{1}),f(Q_{2}))).$$ Therefore, $$\varphi_{f}(L(P_{1})L(P_{2}))=\varphi_{f}(L(P_{1}))\varphi_{f}(L(P_{2})).$$
This completes the proof of the theorem.
\end{proof}
The following result proves that the map $\varphi$ is not in general a bijection.
\begin{lem}
Let $\alpha\in F$ and let $f\in \Inj(F)$. Put
$$P_{n}(X)=\prod_{i=1}^{n}(X-\alpha^{i-1})$$ and
$$\overline{(\alpha,n)}=\{1,\alpha,\ldots,\alpha^{n}\}.$$
Suppose that $\mathbf{card}(\overline{(\alpha,n)})=n+1$ for all $n\in \mathbb{N}$ and let $\psi_{\alpha}: \mathcal{L}_{F}\longrightarrow \mathcal{L}_{F}$ be the injective map defined by
\begin{eqnarray*}
\begin{cases}
\psi_{\alpha}(L(X^{n}))=L(X^{n}) \forall n\in \mathbb{N}\\
\psi_{\alpha}(L((X-\lambda)^{n}))=L(X-f(\lambda))L(P_{n}(X))\\
\hspace*{7.6pc}=\Upsilon(X-f(\lambda),P_{n}(X)) \forall n\in \mathbb{N}\lambda\in F^{\star} \forall n\in \mathbb{N}\\
\psi_{\alpha}(L((X-\lambda_{1})^{n_{1}}\cdots(X-\lambda_{m})^{n_{m}}))=\sum_{i=1}^{m}\psi_{\alpha}(L((X-\lambda_{1})^{n_{i}}))
\end{cases}
\end{eqnarray*}
Then $\psi_{\alpha}\in \Inj(\mathcal{L}_{F})$ if and only if $\textrm{char}(F)=0$.
\end{lem}
\begin{proof} It is easily seen that $L(P_{n}(X))L(P_{m}(X))=L(P_{n+m-1}(X))$.\\
We have $\psi_{\alpha}(L((X-1)^{2}))=L(P_{2}(X))$. Then $$\psi_{\alpha}(L((X-1)^{2\wedge 2}))=L(P_{2}(X))L(P_{2}(X))=L(P_{3}(X)).$$ Hence $2\wedge 2=3$. So $$\psi_{\alpha}(L((X-1)^{3\wedge 3}))=L(P_{3}(X))L(P_{3}(X))=L(P_{5}(X)).$$ Hence $3\wedge 3=5$. A simple induction on $n$ shows that $n\wedge n=2n-1$ for all $n\in \mathbb{N}$. Then $\textrm{char}(F)=0$.
\\Conversly, suppose that $\textrm{char}(F)=0$ which means that $n\wedge n=2n-1$ for all $n\in \mathbb{N}$. Let $P_{1}=X^{n_{1}}Q_{1}\neq1$ and $P_{2}=X^{n_{2}}Q_{2}\neq1$ where $Q_{1}=$ and $Q_{2}$ are polynomials with non-zero constant terms.
\\We have $$\psi_{\alpha}(L(P_{1})L(P_{2}))=\psi_{\alpha}(L(X^{\rho}\Upsilon(Q_{1},Q_{2})))=L(X^{\rho})+\psi_{\alpha}(L(\Upsilon(Q_{1},Q_{2}))).$$
On the other hand, $$\psi_{\alpha}(L(\Upsilon(Q_{1},Q_{2})))=\Sigma_{i\in S_{1}}\varphi_{\alpha}(L(X-\widehat{\Omega}_{i})^{\Omega_{i}^{\xi_{2}}})$$
$$=\Sigma_{i\in S_{1}}L(X-\widetilde{f}(\widehat{\Omega}_{i}))L(P_{\Omega_{i}^{\xi_{2}}}(X))$$
$$=\Sigma_{i\in S_{1}}L(X-\widetilde{f}(\widehat{\Omega}_{i}))L(P_{\Omega_{i}^{\xi_{2}}}(X))$$
Then $$\psi_{\alpha}(L(P_{1})L(P_{2}))=L(X^{\rho})+\Sigma_{i\in S_{1}}L(X-\widetilde{f}(\widehat{\Omega}_{i}))L(P_{\Omega_{i}^{\xi_{2}}}(X))$$
$$\psi_{\alpha}(L(P_{1}))=L(X^{n_{1}})+\Sigma_{\beta(P_{1})\neq0}L(X-\widetilde{f}(\beta))L(P_{\beta(P_{1})}(X))$$
$$\psi_{\alpha}(L(P_{2}))=L(X^{n_{2}})+\Sigma_{\lambda(P_{2})\neq0}L(X-\widetilde{f}(\lambda))L(P_{\lambda(P_{2})}(X))$$
Since $Q_{i}=1$ if and only if $\psi_{\alpha}(L(P_{1}))=L(X^{n_{1}})$, we have
\begin{eqnarray*}
&\psi_{\alpha}(L(P_{1}))\psi_{\alpha}(L(P_{2}))=L(X^{\rho})+\\
\hspace*{1pc}&(\Sigma_{\beta(P_{1})\neq0}L(X-\widetilde{f}(\beta))L(P_{\beta(P_{1})}(X)))(\Sigma_{\lambda(P_{2})\neq0}
L(X-\widetilde{f}(\lambda))L(P_{\lambda(P_{2})}(X)))\\
&=L(X^{\rho})+\Sigma_{\beta(P_{1})\neq0,\lambda(P_{2})\neq0}L(X-\widetilde{f}(\beta)\widetilde{f}(\lambda))L(P_{\beta(P_{1})}(X))L(P_{\lambda(P_{2})}(X))\\
&=L(X^{\rho})+\Sigma_{\beta(P_{1})\neq0,\lambda(P_{2})\neq0}L(X-\widetilde{f}(\beta\lambda))L(P_{\beta(P_{1})+\lambda(P_{2})-1}(X))\\
&=L(X^{\rho})+\Sigma_{i\in S_{1}}L(X-\widetilde{f}(\widehat{\Omega}_{i}))[\sum_{(\beta,\lambda)\in \Omega_{i}} L(P_{\beta(P_{1})\wedge\lambda(P_{2})}(X))]
\end{eqnarray*}
Since $P(n)(X)$ divise $P(m)(X)$ whenever $n\leq m$, we have
$$\sum_{(\beta,\lambda)\in \Omega_{i}} L(P_{\beta(P_{1})\wedge\lambda(P_{2})}(X))=L(P_{\Omega_{i}^{\xi_{2}}}(X)).$$
Thus
\begin{eqnarray*}
&\psi_{\alpha}(L(P_{1}))\psi_{\alpha}(L(P_{2}))=L(X^{\rho})+\\
\hspace*{1pc}&=(\Sigma_{\beta(P_{1})\neq0}L(X-\widetilde{f}(\beta))L(P_{\beta(P_{1})}(X)))(\Sigma_{\lambda(P_{2})\neq0}
L(X-\widetilde{f}(\lambda))L(P_{\lambda(P_{2})}(X)))\\
&=L(X^{\rho})+\Sigma_{i\in S_{1}}L(X-\widetilde{f}(\widehat{\Omega}_{i}))L(P_{\Omega_{i}^{\xi_{2}}}(X))
\end{eqnarray*}
By conclusion, we have $\psi_{\alpha}(L(P_{1})L(P_{2}))=\psi_{\alpha}(L(P_{1}))\psi_{\alpha}(L(P_{2}))$.
\\To complete the proof, it is straightforward to verify that
$$\psi_{\alpha}(L(P_{1})+L(P_{2}))=\psi_{\alpha}(L(P_{1}))+\psi_{\alpha}(L(P_{2}))\quad\text{and}\quad \psi_{\alpha}(L(X-1))=L(X-1);$$
we omit the details here.
\end{proof}

\end{document}